\documentclass{article}
\usepackage {amssymb}
\usepackage {amsmath}
\usepackage {amsthm}
\usepackage{graphicx}
\usepackage {amscd}
\usepackage[colorlinks, citecolor=red]{hyperref}
\newtheorem{thm}{Theorem}
\newtheorem{prop}{Proposition}
\newtheorem{defn}{Definition}
\newtheorem{cor}{Corollary}
\newtheorem{rem}{Remark}
\newtheorem{lem}{Lemma}

\begin{document}
\title{\textbf{Constant scalar curvature K\"{a}hler metric and K-energy}}
\author{Chi Li}
\date{}
\maketitle
\begin{abstract}
\noindent ABSTRACT: Based on Donaldson's method, we prove that, for
an integral K\"{a}hler class, when there is a K\"{a}hler metric of
constant scalar curvature, then it minimizes the K-energy. We do not
assume that the automorphism group is discrete.
\end{abstract}
\begin{section}{Introduction}
Let $X$ be a compact K\"{a}hler manifold of dimension $n$ and fix a
K\"{a}hler class $[\omega]$ on $X$. Define the K\"{a}hler potential
space
\begin{equation}\label{potential}
\mathcal{K}:=\{\phi\;;\;
\omega+\frac{\sqrt{-1}}{2\pi}\partial\bar{\partial}\phi>0\}
\end{equation}
$\overline{\mathcal{K}}:=\mathcal{K}/\mathbb{R}$ is the space of
K\"{a}hler metrics in $[\omega]$. The K-energy functional is defined
by Mabuchi \cite{Ma1}
\begin{defn}
For any $\phi\in\mathcal{K}$, let
$\omega_\phi=\omega+\frac{\sqrt{-1}}{2\pi}\partial\bar{\partial}\phi\in\overline{\mathcal{K}}$,
define
\[
\nu_\omega(\omega_\phi)=-\frac{1}{V}\int_0^1dt\int_X(S(\omega_{\phi_t})-\underline{S})\frac{d\phi_t}{dt}\;\omega_{\phi_t}^n
\]
$\phi_t$ is any path connecting 0 and $\phi$ in $\mathcal{K}$.
$S(\omega_\phi)$ denotes the scalar curvature of K\"{a}hler metric
$\omega_\phi$, and
\[
\underline{S}=\frac{1}{V}\int_XS(\omega)\omega^n=\frac{nc_1(X)\cdot
[\omega]^{n-1}}{[\omega]^n},\quad V=\int_X\omega^n
\]
is the average of scalar curvature, which is independent of chosen
K\"{a}hler metric in $[\omega]$.
\end{defn}
The K-energy is well defined, i.e. it does not depend on the path
connecting 0 and $\phi$. In particular, we can take $\phi(t)=t\phi$.
A K\"{a}hler metric of constant scalar curvature is a critical point
of K-energy and it is a local minimizer.

Now assume the K\"{a}hler class is $c_1(L)\in H^2(X,\mathbb{Z})\cap
H^{1,1}(X,\mathbb{R})$ for some ample line bundle $L$ over $X$.

In this note, we prove the following theorem:
\begin{thm}\label{main}
Suppose that there is a metric $\omega_\infty$ of constant scalar
curvature in the K\"{a}hler class $c_1(L)$. Then $\omega_\infty$
minimizes the K-energy in this K\"{a}hler class.
\end{thm}
In the case of K\"{a}hler-Einstein metrics, this result was proved
by Bando and Mabuchi in \cite{Ba} and \cite{BM}. In \cite{Do2},
Donaldson proved the above theorem under the assumption that the
automorphism group $\mbox{Aut}(X,L)$ is discrete. The theorem was
proved for more general extremal K\"{a}hler metrics on any compact
K\"{a}hler manifolds by Chen-Tian in \cite{CT}, where the authors
used geodesics in infinite dimensional space of K\"{a}hler metrics
in $[\omega]$.

In this note, we use Donaldson's method to prove the theorem for
every integral class case.

The strategy to prove the theorem is finite dimensional
approximation. We sketch the idea here.

Tian's approximation theorem (Proposition \ref{bergexp} and
Corollary \ref{conver}) says that $\mathcal{K}$ can be approximated
by a sequence of finite dimensional symmetric spaces
$\mathcal{H}_k$. Here $\mathcal{H}_k\cong GL(N_k,\mathbb{C})/U(N_k)$
is the space of Hermitian metrics on the complex vector space
$H^0(X, L^k)$.

In \cite{Do2}, Donaldson defined a sequence of functional
$\mathcal{L}_k$ on $\mathcal{K}$ which approximate K-energy as
$k\rightarrow \infty$. When restricted to $\mathcal{H}_k$,
$\mathcal{L}_k$ is bounded below by the logarithmic of Chow norm. It
was known that balanced metric obtains the minimum of Chow norm
(\cite{Zh},\cite{Pa}). In \cite{Do1}, Donaldson already proved, in
the case of discrete automorphism group, existence of balanced
metrics which approximate K\"{a}hler metric of constant scalar
curvature. Putting these together, he can prove the theorem.

Mabuchi \cite{Ma2,Ma3,Ma4} extended many results of \cite{Do1} to
the case where the varieties have infinitesimal automorphisms. As
Mabuchi \cite{Ma6} showed, if the automorphism group is not
discrete, in general there will be no balanced metrics. Instead,
Mabuchi defined T-balanced metrics and T-stability with respect to
some torus group contained in $\mbox{Aut}(X)$. Donaldson claimed
\cite{Do2} one can use these new techniques to prove the above
theorem without assuming that the automorphism group is discrete.

In this note, we use the same quantization strategy. But we don't
need existence of balanced metrics or T-balanced metrics. Instead we
use simpler Bergman metrics constructed directly from
$\omega_\infty$. By suitably normalizing Bergman kernels, we show
that Bergman metrics are almost balanced in an asymptotical sense
(Proposition \ref{estiprop}), and they can help us to prove the
theorem. In this way, we don't need the restriction on the
automorphism group; moreover, our argument is more direct.

As can be seen from the following, the argument follows \cite{Do2}
closely. The idea of using almost balanced metrics is inspired by
works of Mabuchi. In particular, the proof of Lemma \ref{step2} is
inspired by the argument of \cite{Ma5}, page 13. See Remark
\ref{remMa}. As the above work shows, the convexity of various
functionals is the essential property behind the argument.

The author would like to thank: Professor Gang Tian for his constant
encouragement and help; Yanir Rubinstein for pointing out some
inaccuracies in the first version of the note, and for his
friendship and encouragement. The author would also like to thank
all the participants, particularly Yalong Shi, in the geometry
seminar in Peking University. Their patience and encouragement all
help the author to understand the stability condition and work of
Donaldson.
\end{section}
\begin{section}{Notations and preliminaries}
\begin{subsection}{Maps between $\mathcal{K}_k$ and $\mathcal{H}_k$}
We will use some definitions and notations from \cite{Do2}. The set
$\mathcal{K}$ defined in \eqref{potential} depends on reference
K\"{a}hler metric $\omega$. However in the following, we will omit
writing down this dependence, because it's clear that $\mathcal{K}$
is also the set of metrics $h$ on $L$ whose curvature form
\[c_1(L,h):=-\frac{\sqrt{-1}}{2\pi}\partial\bar{\partial}\log h\]
is a positive (1,1) form on $X$. Let $\mathcal{K}_k$ denote the set
of Hermitian metrics on $L^k$ with positive curvature form, then
$\mathcal{K}_k\simeq\mathcal{K}=\mathcal{K}_1$. Let
$N_k=\mbox{dim}\;H^0(X,L^k)$, $V=\int_Xc_1(L)^n$.
We have maps between $\mathcal{K}_k$ and $\mathcal{H}_k$.
\begin{defn}
\begin{eqnarray*}
\mbox{\upshape Hilb}:\mathcal{K}_k&\longrightarrow& \mathcal{H}_k\\
h_k&\mapsto& \|s\|_{\mbox{\scriptsize \upshape
Hilb}(h_k)}^2=\frac{N_k}{Vk^n}\int_X|s|_{h_k}^2\,c_1(L^k,h_k)^n,\quad\forall
s\in
H^0(X,L^k)\\
\mbox{\upshape FS}: \mathcal{H}_k&\longrightarrow& \mathcal{K}_k\\
H_k&\mapsto& |s|_{\mbox{\scriptsize\upshape
FS}(H_k)}^2=\frac{|s|^2}{\sum_{\alpha=1}^{N_k}|s_{\alpha}^{(k)}|^2},\quad\forall
s\in L^k
\end{eqnarray*}
In the above definition, $\{s_{\alpha}^{(k)};1\le\alpha\le N_k\}$ is
an orthonormal basis of the Hermitian complex vector space
$(H^0(X,L^k),H_k)$.
\end{defn}
\end{subsection}
\begin{subsection}{Bergman metrics, expansions of Bergman kernels}
For any fixed K\"{a}hler metric $\omega\in c_1(L)$, take a Hermitian
metric $h$ on $L$ such that $c_1(L,h)=\omega$, the $k$-th Bergman
metric of $h$ is
\[h_{k}=\mbox{\upshape FS}(\mbox{\upshape Hilb}(h^{\otimes k}))\in\mathcal{K}_k\]
Let $\{s_{\alpha}^{(k)},1\le\alpha\le N_k\}$ be an orthonormal basis
of $\mbox{\upshape Hilb}(h^{\otimes k})$. Define the $k$-th
(suitably normalized) Bergman kernel of $\omega$
\[
\rho_k(\omega)=\frac{N_kn!}{Vk^n}\sum_{\alpha=1}^{N_k}|s_{\alpha}^{(k)}|^2_{h^{\otimes
k}}
\]
Note that $h$ is determined by $\omega$ up to a constant, but
$\rho_k(\omega)$ doesn't depend on the chosen $h$.

The following proposition is now well known.
\begin{prop}[\cite{T1},\cite{Cat},\cite{Zel},\cite{Ru},\cite{Lu}]\label{bergexp}
\mbox{}\par
\begin{enumerate}
\item[\upshape (1)]\label{exp} For fixed $\omega$, there is an asymptotic expansion as
$k\rightarrow +\infty$.
\[
\rho_k(\omega)=A_0(\omega)+A_1(\omega)k^{-1}+\dots
\]
where $A_i(\omega)$ are smooth functions on $X$ defined locally by
$\omega$.
\item[\upshape{(2)}]\label{coe} In particular
\[
A_0(\omega)=1,\quad A_1(\omega)=\frac{1}{2}S(\omega)
\]
\item[\upshape (3)] The expansion holds in $C^{\infty}$ in that for any $r, N\ge
0$
\[
\left\|\rho_k(\omega)-\sum_{i=0}^NA_i(\omega)k^{-i}\right\|_{C^r(X)}\le
K_{r,N,\omega}k^{-N-1}
\]
for some constants $K_{r,N,\omega}$. Moreover the expansion is
uniform in that for any $r, N$, there is an integer $s$ such that if
$\omega$ runs over a set of metrics which are bounded in $C^s$, and
with $\omega$ bounded below, the constants $K_{r,N,\omega}$ are
bounded by some $K_{r,N}$ independent of $\omega$.
\end{enumerate}
\end{prop}
\begin{rem}
We take a different normalization of Bergman kernel, so the
expansion starts with order 0, other than order $n$ as it appeared
in (\cite{Do1}, Proposition 6).
\end{rem}
The following approximation result is a corollary of Proposition
\ref{bergexp}.(1)-(2).
\begin{cor}\label{conver}
Using the notation at the beginning of this subsection, we have,
as $k\rightarrow+\infty$, $(h_k)^{\frac{1}{k}}\rightarrow h$, and
$\frac{1}{k}c_1(L^k,h_k)\rightarrow \omega$, the convergence is in
$C^{\infty}$ sense. More precisely, for any $r>0$, there exists a
constant $C_{r,\omega}$ such that
\begin{equation}\label{order}
\left\|\log \frac{h_k^{\frac{1}{k}}}{h}\right\|_{C^r}\le
C_{r,\omega}k^{-2}, \qquad
\left\|\frac{1}{k}c_1(L^k,h_k)-\omega\right\|_{C^r}\le
C_{r,\omega}k^{-2}
\end{equation}
\end{cor}
\begin{proof}
It's easy to see that
\[
(h_k)^{\frac{1}{k}}=h\cdot\left(\sum_\alpha
|s_\alpha^{(k)}|^2_{h^{\otimes k}}\right)^{-\frac{1}{k}}=:h
e^{-\phi_k}
\]
Note that by the expansion in Proposition \ref{bergexp}.(1)-(2), we
have
\begin{eqnarray*}
\sum_\alpha |s_\alpha|_{h^{\otimes
k}}^2&=&\frac{(N_kn!/Vk^n)\sum_{\alpha=1}^{N_k}|s_{\alpha}^{(k)}|_{h^{\otimes
k}}^2}{N_kn!/Vk^n}=\frac{1+\frac{1}{2}S(\omega)k^{-1}+O(k^{-2})}{1+\frac{1}{2}\underline{S}k^{-1}+O(k^{-2})}\\
&=&1+O(k^{-1})
\end{eqnarray*}
So
\[
\phi_k=\frac{1}{k}\log\left(\sum_\alpha |s_\alpha^{(k)}|_{h^{\otimes
k}}^2\right)=O(k^{-2})
\]
The error term is in $C^{\infty}$ sense. So the first inequality in
\eqref{order} holds. The second inequality in \eqref{order} follows
because
\[\frac{1}{k}c_1(L^k,h_k)-\omega=\frac{\sqrt{-1}}{2\pi}\partial\bar{\partial}\phi_k\]
\end{proof}
Now assume we have a K\"{a}hler metric of constant scalar curvature
$\omega_\infty$ in the K\"{a}hler class $c_1(L)$. Take a
$h_{\infty}\in\mathcal{K}_1$ such that
\[
\omega_\infty=c_1(L,h_{\infty})
\]
We will make extensive use of the $k$-th Bergman metric of
$h_\infty$ and its associated objects, so for the rest of note, we
denote
\begin{align}
h_k^*=\mbox{\upshape FS}(\mbox{\upshape Hilb}(h_\infty^{\otimes
k})),\quad \omega_k^*=c_1(L^{k},h_k^*),\quad H_k^*=\mbox{\upshape
Hilb}(h_k^*),\quad h_k^{**}=\mbox{\upshape FS}(\mbox{\upshape
Hilb}(h_k^*))\tag*{$(*)$}\label{hkstar}
\end{align}
Hereafter, we also fix an orthonormal basis $\{\tau_{\alpha}^{(k)},
1\le\alpha\le N_k\}$ of $H_k^*=\mbox{\upshape Hilb}(h_k^*)$.
The next proposition says that Hermitian metrics $h_k^*$ are almost
balanced asymptotically. This will be important for us. (Compare
(\cite{Ma5}, (3.8)-(3.10)))
\begin{prop}\label{estiprop}
For any $r>0$, there exists some constant $C_{r,\omega_\infty}$ such
that
\begin{equation}\label{esti2}
\left\|\sum_{\alpha=1}^{N_k}|\tau_{\alpha}^{(k)}|_{h_k^*}^2-1\right\|_{C^r}\le
C_{r,\omega_\infty}k^{-2}
\end{equation}
So in particular,
\begin{equation}\label{esti3}
\frac{\sqrt{-1}}{2\pi}\partial\bar{\partial}\log\left(\sum_{\alpha=1}^{N_k}|\tau_{\alpha}^{(k)}|^2\right)-\omega_{k}^*=O(k^{-2})
\end{equation}
\end{prop}
\begin{proof}
By proposition \ref{bergexp}, we have
\begin{eqnarray*}\label{esti1}
\sum_{\alpha=1}^{N_k}|\tau_{\alpha}^{(k)}|_{h_k^*}^2-1&=&\frac{(N_kn!/Vk^n)\sum_{\alpha=1}^{N_k}|\tau_{\alpha}^{(k)}|_{h_k^*}^2}{N_kn!/Vk^n}-1=\frac{1+\frac{1}{2}S(\frac{1}{k}\omega_k^*)k^{-1}+O(k^{-2})}{1+\frac{1}{2}\underline{S}k^{-1}+O(k^{-2})}-1\nonumber\\
&=&\left(\frac{1}{2}\left[S\left(\frac{1}{k}\omega_k^*\right)-\underline{S}\right]k^{-1}+O(k^{-2})\right)(1+O(k^{-1}))\nonumber\\
&=&O(k^{-2})
\end{eqnarray*}
Since by Corollary \ref{conver},
$\frac{1}{k}\omega_k^*\rightarrow\omega_\infty$ in $C^{\infty}$, so
$\{\frac{1}{k}\omega_k^*\}$ have bounded geometry. So by Proposition
\ref{bergexp}.(3), the expansions on the first line are uniform in
$k$.

The last equality is because, by \eqref{order}, the convergence rate
of $\frac{1}{k}\omega_k^*\rightarrow\omega_\infty$ is $O(k^{-2})$,
and $S(\omega_\infty)=\underline{S}$, so we also have
$S(\frac{1}{k}\omega_k^*)-\underline{S}=O(k^{-2})$ in $C^{\infty}$.

\eqref{esti3} follows because the left hand side of it is equal to $
\frac{\sqrt{-1}}{2\pi}\partial\bar{\partial}\log\left(\sum_{\alpha=1}^{N_k}|\tau_{\alpha}^{(k)}|_{h_k^*}^2\right)
$
\end{proof}
\end{subsection}
\begin{subsection}{Aubin-Yau functional and Chow norm}\label{aubinchow}
We define the Aubin-Yau functional with respect to $(L^k,
h_{k}^{**})$ by
\[
I_k(h_k^{**}e^{-\phi})=-\int_0^1dt\int_X\frac{d\phi(t)}{dt}\,c_1(L^k,h_k^{**}e^{-\phi(t)})^n
\]
Here $\phi\in\mathcal{K}_k$ i.e. $\frac{1}{k}\phi\in\mathcal{K}$.
$\phi(t)$ is a path connecting 0 and $\phi$ in $\mathcal{K}_k$.

Under the orthonormal basis $\{\tau_{\alpha}^{(k)}, 1\le\alpha\le
N_k\}$ of $H_k^*$. Then $H^0(X,L^k)\cong \mathbb{C}^{N_k}$ and
$\mathbb{P}(H^0(X,L^k)^*)\cong\mathbb{CP}^{N_k-1}$.

For any $H\in\mathcal{K}_k$, take an orthonormal basis $\{s_\alpha,
1\le\alpha\le N_k\}$ of $H$. Let $\det H_k$ denote the determinant
of matrix $(H_k)_{\alpha\beta}=(H_{k}^*(s_\alpha,s_\beta))$.
$\{s_\alpha\}$ determines a projective embedding into
$\mathbb{P}(H^0(X,L^k)^*)\cong\mathbb{CP}^{N_k-1}$. The image of
this embedding is denoted by $X_k(H)\subset\mathbb{CP}^{N_k-1}$ and
has degree $d_k=Vk^n$. $X_k(H)$ has a Chow point
\[
\hat{X}_k(H)\in
W_k:=\{\mbox{Sym}^{d_k}(\mathbb{C}^{N_k})\}^{\otimes(n+1)}
\]
\begin{prop}[\cite{Zh},\cite{Pa}]\label{chownorm}
$W_k$ has a Chow norm $\|\cdot\|_{\mbox{\upshape \scriptsize
CH}(H_{k}^{*})}$, such that
\[
\frac{1}{N_k}\log\det H_k-\frac{1}{Vk^n}I_k(\mbox{\upshape
FS}(H_k))=\frac{1}{Vk^n}\log\|\hat{X}_k(H)\|_{\mbox{\upshape
\scriptsize CH}(H_{k}^{*})}^2
\]
\end{prop}
$SL(N_k,\mathbb{C})$ acts on $\mathcal{H}_k$. Note that $
X_k(\sigma\cdot H_{k}^{*})=\sigma\cdot X_k(H_{k}^*)$. Define
\[
f_k(\sigma)=\log\left(\|\hat{X}_k(\sigma\cdot
H_{k}^{*})\|_{\mbox{\upshape\scriptsize
CH}(H_{k}^{*})}^2\right)\quad\forall \sigma\in SL(N_k,\mathbb{C})
\]
It's easy to see that $f_k(\sigma\cdot \sigma_1)=f_k(\sigma)$ for
any $\sigma_1\in SU(N_k)$, so $f_k$ is a function on the symmetric
space $SL(N_k,\mathbb{C})/SU(N_k)$. We have
\begin{prop}[\cite{KN},\cite{Zh},\cite{Do2}]\label{conv2}
$f_k(\sigma)$ is convex on $SL(N_k,\mathbb{C})/SU(N_k)$.
\end{prop}
To relate $\mathcal{K}_k$ and $\mathcal{H}_k$, following Donaldson
\cite{Do2}, we change $FS(H_k)$ in the above formula into general
$h_k\in\mathcal{K}_k$ and define:
\begin{defn}
For all $h_{k}\in\mathcal{K}_k$ and $H_k\in\mathcal{H}_k$,
\[
\tilde{P}_k(h_{k},H_{k})=\frac{1}{N_k}\log\det
H_{k}-\frac{1}{Vk^n}I_k(h_{k})
\]
\end{defn}
Note that, for any $c\in\mathbb{R}$,
$I_k(e^{c}h_{k})=cVk^n+I_k(h_{k})$, so
\[
\tilde{P}_k(e^{c}h_k,e^{c}H_k)=\tilde{P}_k(h_k,H_k)
\]
\begin{rem}
This definition differs from Donaldson's definition by omitting two
extra terms, since we find of no use for these terms in the
following argument.
\end{rem}
%
\end{subsection}
\end{section}
\begin{section}{Proof of Theorem 1}
\begin{lem}\label{conv1}
For any $h_k$, $h_k'\in \mathcal{K}_k$, with $h_k'=h_k e^{-\phi}$,
we have
\[
-\int_X\phi\;c_1(L^k,h_k)^n\le I_k(h_k')-I_k(h_k)\le -\int_X\phi\;
c_1(L^k, h_k')^n
\]
\end{lem}
This is (\cite{Do2}, Lemma 1).
\begin{proof}
This lemma just says $I_k$ is a convex function on $\mathcal{K}_k$,
regarded as an open subset of $C^{\infty}(X)$, we only need to
calculate its second derivative along the path
$h_k(t)=h_ke^{-t\phi}$:
\[
\frac{d^2}{dt^2}I_k(h_k)=-\int_X\phi\triangle_t \phi
c_1(L^k,h_k(t))^n=\int_X|\nabla_t\phi|^2 c_1(L^k,h_k(t))^n\ge 0
\]
$\triangle_t$ and $\nabla_t$ is the Laplace and gradient operator of
K\"{a}hler metric $c_1(L^k,h_k(t))$.
\end{proof}
From now on, fix a $\omega\in c_1(L)$, take a Hermitian metric
$h\in\mathcal{K}$ such that $\omega=c_1(L,h)$. We have the $k$-th
Bergman metric $h_k=\mbox{FS}(\mbox{Hilb}(h^{\otimes k}))$ and
corresponding K\"{a}hler metric $\omega_{k}=c_1(L^k,h_{k})$ by
Corollary \ref{conver}
\[
(h_k)^{\frac{1}{k}}\rightarrow h,\quad
\frac{1}{k}\omega_{k}\rightarrow \omega,\quad \mbox{in}\quad
C^{\infty}
\]
\begin{lem}\label{step1}
\[
\tilde{P}_k(h_k,\mbox{\upshape
Hilb}(h_k))\ge\tilde{P}_k(\mbox{\upshape FS}(\mbox{\upshape
Hilb}(h_k)),\mbox{\upshape Hilb}(h_k))
\]
\end{lem}
This is a corollary of (\cite{Do2}, Lemma 4). Since the definition
of $\tilde{P}$ is a little different from that in \cite{Do2}, we
give a direct proof here.
\begin{proof}
Let $h_k'=\mbox{FS}(\mbox{Hilb}(h_k))$. Then
\[
\tilde{P}_k(h_k,\mbox{\upshape
Hilb}(h_k))-\tilde{P}_k(\mbox{\upshape FS}(\mbox{\upshape
Hilb}(h_k)),\mbox{\upshape
Hilb}(h_k))=\frac{1}{Vk^n}(I(h_k')-I(h_k))
\]
Let $\{s_{\alpha}^{(k)}, 1\le\alpha\le N_k\}$ be an orthonormal
basis of $\mbox{Hilb}(h_k)$. Then $
\log\frac{h_k'}{h_k}=-\log(\sum_{\alpha=1}^{N_k}|s_{\alpha}^{(k)}|_{h_k}^2)
$. By Lemma \ref{conv1} and concavity of the function $\log$,
\begin{eqnarray*}
\frac{1}{Vk^n}(I(h_k')-I(h_k))&\ge&
-\frac{1}{Vk^n}\int_X\log\left(\sum_{\alpha}|s_{\alpha}^{(k)}|_{h_k}^2\right)c_1(L^k,h_k)^n\\
&\ge&-\log\left(\frac{1}{N_k}\frac{N_k}{Vk^n}\int_X\sum_{\alpha}|s_{\alpha}^{(k)}|^2_{h_k}c_1(L^k,h_k)^n\right)\\
&=&-\log\left(\frac{1}{N_k}\sum_{\alpha}\|s_{\alpha}^{(k)}\|^2_{\mbox{\scriptsize
\upshape Hilb}(h_k)}\right)=0
\end{eqnarray*}
\end{proof}
\begin{lem}\label{step2}
There exits a constant $C>0$, depending only on $h$ and $h_\infty$,
such that
\[
\tilde{P}_k(\mbox{\upshape FS}(\mbox{\upshape
Hilb}(h_k)),\mbox{\upshape Hilb}(h_k))-\tilde{P}_k(\mbox{\upshape
FS}(H_k^*),H_k^*)\ge -Ck^{-1}
\]
\end{lem}
\begin{proof}
Recall that $H_k^*=\mbox{\upshape Hilb}(h_{k}^*)$ and
$\{\tau_{\alpha}^{(k)}; 1\le\alpha\le N_k\}$ is an orthonormal basis
of $H_k^*$ (See \ref{hkstar}). Let $H_k=\mbox{\upshape Hilb}(h_{k})$
and $\{s_\alpha^{(k)}; 1\le\alpha\le N_k\}$ be an orthonormal basis
of $H_k$. Transforming by a matrix in $SU(N_k)$, we can assume
\[
s_\alpha^{(k)}=e^{\lambda^{(k)}_\alpha}\tau_{\alpha}^{(k)}
\]
\begin{equation}\label{lambda}
e^{-2\lambda_\alpha^{(k)}}=\frac{N_k}{Vk^n}\int_{X}|\tau_{\alpha}^{(k)}|_{h_{k}}^2\omega_{k}^{n}
\end{equation}
Since by Corollary \ref{conver} we have the following uniform
convergence in $C^{\infty}$: $(h_k)^{\frac{1}{k}}\rightarrow h$,
$\frac{1}{k}\omega_k\rightarrow\omega$,
$(h_k^*)^{\frac{1}{k}}\rightarrow h_\infty$,
$\frac{1}{k}\omega_k^*\rightarrow\omega_\infty$, There exists a
constant $C_1>0, C_2>0$, depending only on $h$ and $h_\infty$, such
that $C_1^{-k}\le\frac{h_k}{h_k^*}\le C_1^k$,
$C_2^{-1}\omega_k^*\le\omega_k\le C_2\omega_k^*$, so we see from
\eqref{lambda} that $ |\lambda_\alpha^{(k)}|\le Ck $.

Let
$\underline{\lambda}=\frac{1}{N_k}\sum_{\beta=1}^{N_k}\lambda_\beta^{(k)}$,
$H_k'=e^{2\underline{\lambda}}H_k$,
$\hat{\lambda}_\alpha^{(k)}=\lambda_\alpha^{(k)}-\underline{\lambda}$.
Then
$\{\hat{s}_\alpha^{(k)}=e^{\hat{\lambda}_{\alpha}^{(k)}}\tau_{\alpha}^{(k)}\}$
is an orthonormal basis of $H_k'$. Note that
$\hat{\lambda}_\alpha^{(k)}$ satisfies the same estimate as
$\lambda_\alpha^{(k)}$:
\begin{equation}\label{estilam2}
|\hat{\lambda}_\alpha^{(k)}|\le Ck
\end{equation}
$\hat{\Lambda}_{\alpha\beta}=\hat{\lambda}_\alpha^{(k)}\delta_{\alpha\beta}$
is a diagonal matrix in $SL(N_k,\mathbb{C})$. By scaling invariant
of $\tilde{P}_k$ and proposition \ref{chownorm}, we have
\[
\tilde{P}_k(\mbox{\upshape
FS}(H_{k}),H_{k})=\tilde{P}_k(\mbox{\upshape
FS}(H'_k),H'_k)=\frac{1}{Vk^n}\log\|\hat{X}_{k}(H'_k)\|_{\mbox{\upshape
\scriptsize CH}(H_{k}^{*})}^2
\]
\[
\tilde{P}_k(\mbox{\upshape
FS}(H^*_k),H^*_k)=\frac{1}{Vk^n}\log\|\hat{X}_{k}(H_k^*)\|_{\mbox{\upshape
\scriptsize CH}(H_{k}^{*})}^2
\]
As in Section \ref{aubinchow}, let
\[
X_k(s)=e^{s\hat{\Lambda}}\cdot X_k(H_k^*)
\]
\[
f_k(s)=\log\|\hat{X}_k(s)\|^2_{\mbox{\upshape \scriptsize
CH}(H_k^*)}
\]
By Proposition \ref{conv2}, $f_k(s)$ is a convex function of $s$, so
\[
f_k(1)-f_k(0)\ge f'(0)
\]
We can estimate $f_k'(0)$ by the estimates in Proposition
\ref{estiprop}:
\begin{eqnarray*}
f_k'(0)&=&\int_{X}\frac{\sum_{\alpha}\hat{\lambda}_\alpha^{(k)}|\tau_{\alpha}^{(k)}|^2}{\sum_\alpha|\tau_{\alpha}^{(k)}|^2}\left(\frac{\sqrt{-1}}{2\pi}\partial\bar{\partial}
\log\sum_{\alpha=1}^{N_k}|\tau_{\alpha}^{(k)}|^2\right)^{n}\\
&=&\int_X\frac{\sum_{\alpha}\hat{\lambda}_\alpha^{(k)}|\tau_{\alpha}^{(k)}|^2}{1+O(k^{-2})}(1+O(k^{-2}))\omega_{k}^{*n}\\
&=&\int_XO(k^{-2})(\sum_{\alpha=1}^{N_k}\hat{\lambda}_\alpha^{(k)}|\tau_{\alpha}^{(k)}|^2)\omega_k^{*n}
\end{eqnarray*}
where the last equality is because of
\[
\int_{X}\sum_{\alpha=1}^{N_k}\hat{\lambda}^{(k)}_\alpha|\tau_{\alpha}^{(k)}|^2_{h_k^*}\omega_k^{*n}=\frac{Vk^n}{N_k}\sum_{\alpha=1}^{N_k}\hat{\lambda}^{(k)}_\alpha=0
\]
By the estimate for $\hat{\lambda}_\alpha^{(k)}$ \eqref{estilam2},
we get
\[
|f_k'(0)|\le C k^{-2} k N_k\le Ck^{n-1}
\]
So $f(1)-f(0)\ge f_k'(0)\ge -Ck^{n-1}$, and
\[\frac{1}{Vk^n}(\log\|\hat{X}_k(H_k)\|_{\mbox{\upshape \scriptsize CH}}^2-\frac{1}{Vk^n}\log\|\hat{X}_{k}(H_k^*)\|_{\mbox{\upshape \scriptsize CH}}^2)
=\frac{1}{Vk^n}(f_k(1)-f_k(0))\ge -C\frac{1}{Vk^n}k^{n-1}\ge
-Ck^{-1}
\]
\end{proof}
\begin{rem}\label{remMa}
The proof of this lemma is similar to the argument in the beginning
part of (\cite{Ma5}, Section 5) where Mabuchi proved semi-stability
of varieties with constant scalar curvature metrics. Roughly
speaking, here we consider geodesic segment connecting $H_k^*$ and
$H_k$ in $\mathcal{H}_k$, while Mabuchi (\cite{Ma5}, Section 5)
considered geodesic ray in $\mathcal{H}_k$ defined by a test
configuration. The estimates in Proposition \ref{estiprop} show
that, to prove the semi-stability as in Mabuchi's argument
(\cite{Ma5}, Section 5), we only need Bergman metrics of $h_\infty$
instead of Mabuchi's T-balanced metrics.
\end{rem}
\begin{rem}
In (\cite{Do2}, Corollary 2), $H_k^*$ is taken to be balance metric,
i.e. $H_k^*$ is a fixed point of the mapping $\mbox{\upshape
Hilb}(\mbox{\upshape FS}(\cdot))$. Then the difference in the Lemma
\ref{step2} is nonnegative, instead of bounded below by error term
$-Ck^{-1}$. Similar remark applies to the following lemma: the
difference in Lemma \ref{step3} equals to 0 for balanced metric.
\end{rem}
\begin{lem}\label{step3}
There exists a constant $C>0$, which only depends on $h_\infty$,
such that
\[\left|\tilde{P}_k(\mbox{\upshape FS}(H_k^*),H_k^*)-\tilde{P}_k(h_k^*,\mbox{\upshape Hilb}(h_k^*))\right|\le Ck^{-2}\]
\end{lem}
\begin{proof}
Recall that \ref{hkstar}
\[
H_k^*=\mbox{\upshape Hilb}(h_k^*),\quad h_k^{**}=\mbox{\upshape
FS}(H_k^*)=\mbox{\upshape FS}(\mbox{\upshape Hilb}(h_k^*))
\]
It's easy to see that
\begin{equation*}
\tilde{P}_k(\mbox{\upshape
FS}(H_k^*),H_k^*)-\tilde{P}_k(h_k^*,\mbox{\upshape
Hilb}(h_k^*))=\frac{1}{Vk^n}(I_k(h_k^*)-I_k(h_k^{**}))
\end{equation*}
For any section $s$ of $L^k$, $
|s|_{h_k^{**}}^2=\frac{|s|_{h_k^*}^2}{\sum_\alpha|\tau_{\alpha}^{(k)}|_{h_k^*}^2}
$. So
\[
\frac{h_k^{*}}{h_k^{**}}=\sum_{\alpha=1}^{N_k}|\tau_{\alpha}^{(k)}|_{h_k^*}^2 
\]
By proposition \ref{estiprop}.
$\left|\log\frac{h_k^{**}}{h_k^{*}}\right|=|\log(1+O(k^{-2}))|=O(k^{-2})$,
So by Lemma \ref{conv1}, we get
\[
\left|\frac{1}{Vk^n}(I_k(h_k^*)-I_k(h_k^{**}))\right|\le Ck^{-2}
\]
\end{proof}
Take any K\"{a}hler metric
$\omega_\phi=c_1(L,h)+\frac{\sqrt{-1}}{2\pi}\partial\bar{\partial}\phi\in
[\omega]$. Let $h_k(\phi)=h^{\otimes k}e^{-k\phi}$. Define
\[
\mathcal{L}_k(\omega_\phi)=\tilde{P}_k(h_k(\phi),\mbox{\upshape
Hilb}(h_k(\phi)))
\]
\begin{lem}\label{approx}
There exist constants $\mu_k$, such that
\[\mathcal{L}_k(\omega_\phi)+\mu_k=\frac{1}{2}\nu_\omega(\omega_\phi)+O(k^{-1})\]
Here $O(k^{-1})$ depend on $\omega$ and $\omega_\phi$.
\end{lem}
\begin{proof}
Let $\phi(t)=t\phi\in\mathcal{K}$, $h_k(t)=h_k e^{-tk\phi}$,
$\omega_{\phi_t}=\omega+t\frac{\sqrt{-1}}{2\pi}\partial\bar{\partial}\phi$,
$\triangle_t$ be the Laplace operator of metric $\omega_{\phi_t}$.
Plugging in expansions for Bergman kernels $\rho_k$ in Proposition
\ref{bergexp}, we get
\begin{eqnarray*}
\frac{d}{dt}\tilde{P}_k(h_k(t),\mbox{\upshape
Hilb}(h_k(t)))&=&\frac{1}{N_kn!}\int_X\frac{N_kn!}{Vk^n}\sum_{\alpha}|s_\alpha^{(k)}|^2_{h_k}(-k\phi+\triangle_t\phi)k^n\omega_\phi^n+\frac{1}{Vk^n}\int_Xk
\phi k^n\omega_\phi^n\\
&=&\frac{1}{V}\frac{k^n}{k^n+\frac{1}{2}\underline{S}\,k^{n-1}+\cdots}\int_X(-k\rho_k+\triangle_t\rho_k)\phi\omega_\phi^n+\frac{k}{V}\int_X\phi\omega_\phi^n\\
&=&-\frac{1}{2V}\int_X(S(\omega_{\phi_t})-\underline{S})\phi\omega_\phi^n+O(k^{-1})
\end{eqnarray*}
$\{\omega_{\phi_t},0\le t\le1\}$ have uniformly bounded geometry, so
by Proposition \ref{bergexp}.(3), the expansions above are uniform.
So the Lemma follows after integrating the above equation.
\end{proof}
\begin{proof}{[Proof of Theorem \ref{main}]}
By Lemma \ref{step1}, Lemma \ref{step2}, Lemma \ref{step3}
\begin{eqnarray*}
\tilde{P}_k(h_k,\mbox{\upshape Hilb}(h_k))&\ge&
\tilde{P}_k(\mbox{\upshape FS}(\mbox{\upshape
Hilb}(h_k)),\mbox{\upshape Hilb}(h_k))\\
&\ge&\tilde{P}_k(\mbox{\upshape FS}(\mbox{\upshape
Hilb}(h_k^*)),\mbox{\upshape Hilb}(h_k^*))+O(k^{-1})\\
&=&\tilde{P}_k(h_k^*,\mbox{\upshape Hilb}(h_k^*))+O(k^{-1})
\end{eqnarray*}
So by Lemma \ref{approx}
\begin{eqnarray*}
\nu_\omega(\omega_\phi)&=&2\mathcal{L}_k(\omega_\phi)+2\mu_k+O(k^{-1})=2\tilde{P}_k(h_k,\mbox{\upshape
Hilb}(h_k))+2\mu_k+O(k^{-1})\\
&\ge&2\tilde{P}_k(h_k^*,\mbox{\upshape
Hilb}(h_k^*))+2\mu_k+O(k^{-1})=2\mathcal{L}_k\left(\frac{1}{k}\omega_k^*\right)+2\mu_k+O(k^{-1})\\
&=&\nu_\omega\left(\frac{1}{k}\omega_k^*\right)+O(k^{-1})\\
&=&\nu_\omega(\omega_\infty)+O(k^{-1})
\end{eqnarray*}
The last line is because
$\frac{1}{k}\omega_k^*\rightarrow\omega_\infty$ in $C^{\infty}$. The
Theorem follows by letting $k\rightarrow+\infty$.
\end{proof}
\end{section}

\end{document}